%% file: rsvd_cnstcy_013020.tex
\documentclass[preprint,authoryear,p3]{elsarticle}

\usepackage{amsthm,amssymb,amsmath,amsfonts,bm,epsfig,epstopdf,eqnarray}
\usepackage{graphicx,psfrag,subfig}
\usepackage{algorithm}
\usepackage{algorithmic}
\usepackage[usenames,dvipsnames]{color}
\usepackage{stmaryrd}
\usepackage{natbib}
\usepackage{ulem}
\usepackage{marginnote}
\usepackage[displaymath, mathlines]{lineno}
\usepackage{environ}
\usepackage[yyyymmdd,hhmmss]{datetime}

\include{bold}

\newtheorem{thm}{Theorem}[section]
\newtheorem{cor}[thm]{Corollary}

\def\intercal{\top}

\newcommand{\LeftEqNo}{\let\veqno\@@leqno}

\journal{Statistics and Probability Letter}
\begin{document}

\begin{frontmatter}



\title{A Consistency Theorem for Randomized Singular Value Decomposition}


\author[Chen]{Ting-Li Chen\corref{cor1}}
\ead{tlchen@stat.sinica.edu.tw}

\author[Chen]{Su-Yun Huang}
\ead{syhuang@stat.sinica.edu.tw}

\author[Wang]{Weichung Wang}
\ead{wwang@ntu.edu.tw}

\cortext[cor1]{Corresponding author}
\address[Chen]{Institute of Statistical Science,
Academia Sinica, Taipei 11529, Taiwan}

\address[Wang]{Institute of Applied Mathematical Sciences,
National Taiwan University, Taipei 10617, Taiwan}

\begin{abstract}
The singular value decomposition (SVD) and the principal component analysis are fundamental tools and probably the most popular methods for data dimension reduction. The rapid growth in the size of data matrices has lead to a need
for developing efficient large-scale SVD algorithms. Randomized SVD was proposed, and its potential was demonstrated for computing a low-rank SVD \citep{rokhlin2009randomized}. In this article, we provide a consistency theorem for the randomized SVD algorithm and a numerical example to show how the random projections to low dimension affect the consistency.
\end{abstract}


\end{frontmatter}

\section{Introduction}

The singular value decomposition (SVD) and the principal component analysis are fundamental and probably the most popular data analytic methods for dimension reduction. Consider the rank-$k$ SVD of a given ${m\times n}$ real matrix
\begin{equation}
\Ab =\Ub\Sigmab\Vb^\intercal
 \approx {\Ub}_{k}\, {\Sigmab}_{k}\,{\Vb}_{k}^\intercal,
\label{eq:rank_k_svd}
\end{equation}
where $\Ub\Sigmab\Vb^\intercal$ is the full SVD and ${\Ub}_{k}\,
{\Sigmab}_{k}\,{\Vb}_{k}^\intercal$ is the truncated rank-$k$ SVD.
The columns of ${\Ub}_{k}$ and ${\Vb}_{k}$ are the $k$ leading left
and right singular vectors of $\Ab$, respectively. The diagonal
entries of ${\Sigmab}_{k}$ are the $k$ largest singular values of
$\Ab$. The computational complexity of the rank-$k$ SVD of $\Ab$ is $O(kmn)$.
When $m$ and $n$ are both large, the computational load is quite heavy.
In recent years, randomized algorithms have been emerging from the community of scientific computing to get fast solutions for big data and big matrix analysis \citep{rokhlin2009randomized,mahoney2011randomized,woodruff2014sketching,slide:RandNLA}. In \cite{rokhlin2009randomized}, a randomized algorithm for rank-$k$ SVD and rank-$k$ PCA for a large matrix was proposed. This randomized SVD (rSVD) algorithm has used a random mapping to bring $\Ab$ to a reduced matrix in low dimension. Next, the SVD of this reduced matrix in low dimension is performed and then mapped back to the
original space to obtain an approximate truncated SVD of $\Ab$. The rSVD algorithm is summarized below.
\begin{algorithm}
  \caption{rSVD \citep{rokhlin2009randomized}}
  \label{alg:rsvd}
  \begin{algorithmic}[1]
    \REQUIRE $\Ab$ (real $m \times n$ matrix, without loss of generality assume $m\le  n$), $k$ (desired rank of truncated SVD), $p$ (parameter for oversampling dimension),
    $\ell=k+p<m$ (dimension of the approximate column space),
    $q$ (exponent of the power method)
    \ENSURE Approximate rank-$k$ SVD of $\Ab \approx \widehat{\Ub}_k\,
        \widehat{\Sigmab}_k \, \widehat{\Vb}_k^\intercal$
    \STATE Generate an $n \times \ell$ random matrix $\Omegab$.
    \STATE Assign $\Yb \leftarrow (\Ab\Ab^\intercal)^{q}\Ab\Omegab$.
    \STATE Compute $\Qb$ whose columns are an orthonormal basis of $\Yb$.
    \STATE Compute the SVD of $\Qb^\intercal\Ab = \widehat{\Wb}_{\ell}\,
           \widehat{\Sigmab}_{\ell}\, \widehat{\Vb}_{\ell}^\intercal$.
    \STATE Assign $\widehat{\Ub}_{\ell} \leftarrow \Qb\widehat{\Wb}_{\ell}$.
    \STATE Extract the leading $k$ singular vectors and singular values from $\widehat{\Ub}_{\ell}$, $\widehat\Sigmab_\ell$
    and $\widehat{\Vb}_{\ell}$ to obtain $\widehat{\Ub}_k$, $\widehat\Sigmab_k$ and $\widehat{\Vb}_k$.
  \end{algorithmic}
  \end{algorithm}

In this article, we provide the rSVD algorithm a statistical justification. We show that
$\Qb$, which is the computed orthonormal basis of the reduced matrix $\Yb$, contains consistent information of the left singular vectors $\Ub$ in the sense of Theorem~\ref{key}.

\section{Main theorem}
We state and prove the consistency theorem in this section. This theorem suggests that $\Qb$ contains consistent information for the left singular vectors $\Ub$ and for the corresponding decreasing order in terms of the singular values.
\begin{thm}\label{key}
For a given matrix $\Ab$ with strictly decreasing singular values, let $\Qb$ be the orthonormal basis of the random
subspace computed by Algorithm~\ref{alg:rsvd} with $\Omegab$
having i.i.d. entries from the standard Gaussian distribution. Then, we have
\begin{equation}
E (\Qb\Qb^\intercal ) = \Ub\Lambdab\Ub^\intercal,
\label{eq:EQQT}
\end{equation}
where (a) $\Lambdab$ is a diagonal matrix,
(b)~each of the diagonal entries is in the interval $(0,1)$, and
(c)~these diagonal entries are strictly decreasing.
\end{thm}

\begin{proof}
Without loss of generality, we assume the exponent in Step 2 of the rSVD algorithm  $q=0$. If not, re-assign $(\Ab\Ab^\top)^q\Ab$ to $\Ab$ and proceed with the following proof. Note that $\Sigmab$ is a diagonal matrix of size $m\times n$ with $m\le n$,
\[
\Sigmab =\left[\begin{array}{ccccccc}
\sigma_1&0&\cdots& \cdots &\cdots& \cdots &0\\
0& \sigma_2 &0&\cdots & \cdots& \cdots &0 \\
\vdots & \ddots &\ddots & \ddots& \cdots & \cdots &0\\
0& \cdots & \cdots  & \sigma_m &0&\cdots&0\\
\end{array} \right].
\]
We have the expectation
\begin{eqnarray*}
&& E\left(\Qb\Qb^\intercal\right)
   = E\left(\Ab\Omegab \left(\Omegab^\intercal\Ab^\intercal \Ab\Omegab\right)^{-1}
 \Omegab^\intercal\Ab^\intercal\right)\nonumber\\
&=& \Ub~ E\left(\Sigmab\Vb^\intercal\Omegab
  \left(\Omegab^\intercal\Vb\Sigmab^\top\Sigmab\Vb^\intercal\Omegab\right)^{-1}
  \Omegab^\intercal\Vb\Sigmab^\top\right)~ \Ub^\intercal
  =\Ub\Lambdab\Ub^\intercal,
\end{eqnarray*}
where
\begin{equation}\label{eq:Lambda}
\Lambdab=E\left(\Sigmab\Vb^\intercal\Omegab
  \left(\Omegab^\intercal\Vb\Sigmab^\top\Sigmab\Vb^\intercal\Omegab\right)^{-1}
  \Omegab^\intercal\Vb\Sigmab^\top\right).
\end{equation}
Note that $\Omegab^\intercal\Vb\Sigmab^\top\Sigmab
\Vb^\intercal\Omegab$ is non-singular with probability one, as $\Omegab$ consists of i.i.d. Gaussian entries.

(a)~{\it First, we show that $\Lambdab$ is a diagonal matrix.} Let $\Zb_{\ell\times n} =\left[\zb_1,\dots,\zb_n \right] =
\Omegab^\intercal\Vb$. That is, $\zb_j
=\Omegab^\intercal \vb_j$, where $\vb_j$ is the $j$th column
of~$\Vb$. The
$(j,j')$th entry of $\Lambdab$ is given by
\[E\left(\sigma_j\sigma_{j'}\zb_j^\intercal
 \Big(\sum_{l=1}^m \sigma_l^2 \zb_l\zb_l^\intercal\Big)^{-1}
 \zb_{j'}\right),\]
where $\sigma_k$ is the $k$th diagonal element of $\Sigmab$.
Below we show that all off-diagonal entries of
$\Lambdab$ are zero. Without loss of generality, consider the
$(1,j)$th entry of $\Lambdab$. Let
$\Vb_{-1}=\left[-\vb_1,\vb_2,\ldots,\vb_n\right]$ and consider
\[\widetilde{\Omegab} = \Vb \Vb_{-1}^\intercal \Omegab.\]
Then, $\widetilde{\Omegab}^{\intercal}\Vb=\Omegab^\intercal\Vb_{-1}
\Vb^\intercal \Vb = \Omegab^\intercal\Vb_{-1}$. Let
$\left[\widetilde\zb_1,\dots,\widetilde\zb_n \right] =
\widetilde{\Omegab}^{\intercal}\Vb$. Then,
$\widetilde\zb_1=\widetilde{\Omegab}^{\intercal}\vb_1=-\zb_1$ and
$\widetilde\zb_j=\widetilde{\Omegab}^{\intercal}\vb_j=\zb_j,
~\forall j\neq 1$. Since $\Omegab$
consists of i.i.d. Gaussian entries and $\Vb
\Vb_{-1}^\intercal$ is an $n\times n$ orthogonal matrix,
$\Omegab$ and $\widetilde{\Omegab}$ have the same distribution. It implies
that $\left[\widetilde\zb_1,\dots,\widetilde\zb_n \right]$ and
$\left[\zb_1,\dots,\zb_n \right]$ follow the same distribution. Then,
\[
\zb_1^\intercal \Big(\sum_{l=1}^m \sigma_l^2 \zb_l\zb_l^\intercal \Big)^{-1}\zb_j
\stackrel{d}{=}\widetilde\zb_1^\intercal \Big(\sum_{l=1}^m \sigma_l^2
  \widetilde\zb_l\widetilde\zb_l^\intercal\Big)^{-1}\widetilde\zb_j
= -\zb_1^\intercal
 \Big(\sum_{l=1}^m \sigma_l^2 \zb_l\zb_l^\intercal \Big)^{-1}\zb_j,
\]
where ``$\stackrel{d}=$'' indicates ``equal in distribution''. Therefore, for
the $(1,j)$th entry of $\Lambdab$, we have
\[
E\left\{\zb_1^\intercal \Big(\sum_{l=1}^m \sigma_l^2
  \zb_l\zb_l^\intercal \Big)^{-1}\zb_{j}\right\}
= -E\left\{\zb_1^\intercal \Big(\sum_{l=1}^m \sigma_l^2
  \zb_l\zb_l^\intercal \Big)^{-1}\zb_{j}\right\}= 0.
\]

(b) {\it Next, we show that all the diagonals, $E\left(\sigma_j^2
\zb_j^\intercal \Big(\sum_{l=1}^m \sigma_l^2
\zb_l\zb_l^\intercal\Big)^{-1} \zb_{j}\right)$, $j=1,\dots,m$, are
less than one.} Let $\Bb_{(-j)} = \sum_{l\neq j}^m \sigma_l^2
\zb_l\zb_l^\intercal$. As $\Omegab$ consists of i.i.d. Gaussian
entries and $\ell<m$, $\Bb_{(-j)}$ is strictly positive definite
with probability one. By Sherman-Morrison-Woodbury matrix identity,
we have
\[
\Big(\sum_{l=1}^m \sigma_l^2 \zb_l\zb_l^\intercal\Big)^{-1}
=
  \Bb_{(-j)}^{-1} - \frac {\sigma_j^2  \Bb_{(-j)}^{-1} \zb_j \zb_j^\intercal
  \Bb_{(-j)}^{-1}}{1+\sigma_j^2 \zb_j^\intercal \Bb_{(-j)}^{-1} \zb_j}.
  \]
Then,
\begin{eqnarray}
&&\sigma_j^2 \zb_j^\intercal \Big(\sum_{l=1}^m \sigma_l^2 \zb_l\zb_l^\intercal\Big)^{-1} \zb_j
  = \sigma_j^2 \zb_j^\intercal \left(\Bb_{(-j)}^{-1} - \frac {\sigma_j^2  \Bb_{(-j)}^{-1} \zb_j \zb_j^\intercal
    \Bb_{(-j)}^{-1}} {1+\sigma_j^2 \zb_j^\intercal \Bb_{(-j)}^{-1} \zb_j}\right) \zb_j\nonumber\\
&=& \sigma_j^2 \left(\zb_j^\intercal \Bb_{(-j)}^{-1} \zb_j -\frac {\sigma_j^2 \zb_j^\intercal\Bb_{(-j)}^{-1}
   \zb_j \zb_j^\intercal \Bb_{(-j)}^{-1} \zb_j} {1+\sigma_j^2 \zb_j^\intercal \Bb_{(-j)}^{-1} \zb_j}\right)\nonumber\\
   &=&\frac {\sigma_j^2 \zb_j^\intercal \Bb_{(-j)}^{-1} \zb_j} {1+\sigma_j^2 \zb_j^\intercal \Bb_{(-j)}^{-1} \zb_j}
    =1-\frac 1 {1+\sigma_j^2 \zb_j^\intercal \Bb_{(-j)}^{-1} \zb_j} <1.
\label{eq:lt_one}
\end{eqnarray}
By taking expectation, we have $E\big\{\sigma_j^2 \zb_j^\intercal (\sum_{l=1}^m \sigma_l^2 \zb_l\zb_l^\intercal)^{-1} \zb_j\big\} <1$.

(c)~{\it Finally, we want to show that $E\big\{\sigma_j^2
\zb_j^\intercal \big(\sum_{l=1}^m \sigma_l^2
\zb_l\zb_l^\intercal\big)^{-1} \zb_{j}\big\}$ is strictly decreasing
as $j$ increases.} Without loss of generality, we will only show the
comparison for $j=1,2$, i.e., we would like to establish the following inequality.
\[E\left\{\sigma_1^2 \zb_1^\intercal
\big(\sum_{l=1}^m \sigma_l^2 \zb_l\zb_l^\intercal\big)^{-1}
\zb_{1}\right\} > E\left\{\sigma_2^2 \zb_2^\intercal \big(\sum_{l=1}^m
\sigma_l^2 \zb_l\zb_l^\intercal\big)^{-1} \zb_{2}\right\}.\]
Consider
$\widetilde{\Omegab}=\Vb\Vb_{1,2}^\intercal \Omegab$, where
$\Vb_{1,2}=\left[\vb_2,\vb_1,\vb_3,\ldots,\vb_n\right]$. Note that $\widetilde\Omegab$ and $\Omegab$ have the same distribution.
Let $\widetilde{\Omegab}^\intercal \Vb = \left[\xb_1,\xb_2,\dots,\xb_n \right]$. Then,  $\xb_1=\zb_2$,
$\xb_2=\zb_1$, and $\xb_j=\zb_j$ for all $3 \leq j \leq n$. Similar
to (\ref{eq:lt_one}), for $j=1,\dots,m$,
\begin{equation}\label{eq:lt_two}
\sigma_j^2 \xb_j^\intercal \Big(\sum_{l=1}^m \sigma_l^2 \xb_l\xb_l^\intercal\Big)^{-1} \xb_j
 = 1-\frac 1 {1+\sigma_j^2 \xb_j^\intercal \widetilde{\Bb}_{(-j)}^{-1} \xb_j},
\end{equation}
where $\widetilde{\Bb}_{(-j)} = \sum_{l\neq j}^m
\sigma_l^2\xb_l\xb_l^\intercal$. Again, we only need to consider the
case that $\widetilde{\Bb}_{(-j)}$ is of full rank, which holds with
probability one. Observe that
$\widetilde{\Bb}_{(-2)}=\Bb_{(-1)}+(\sigma_1^2-\sigma_2^2) \zb_2
\zb_2^\intercal$. Then,
\begin{eqnarray*}
&&\xb_2^\intercal \widetilde{\Bb}_{(-2)}^{-1} \xb_2 = \zb_1^\intercal \left(\Bb_{(-1)}
 +(\sigma_1^2-\sigma_2^2) \zb_2 \zb_2^\intercal\right)^{-1} \zb_1 \\
&=& \zb_1^\intercal \Bb_{(-1)}^{-1} \zb_1  -\frac { (\sigma_1^2-\sigma_2^2)\zb_1^\intercal\Bb_{(-1)}^{-1}
   \zb_2 \zb_2^\intercal \Bb_{(-1)}^{-1} \zb_1} {1+ \zb_{2}^\intercal \Bb_{(-1)}^{-1} \zb_{2}}
  \leq \zb_1^\intercal \Bb_{(-1)}^{-1} \zb_1.
\end{eqnarray*}
The equality holds only when $\zb_1^\intercal\Bb_{(-1)}^{-1}
\zb_2=0$, which happens with zero probability.
Since $\sigma_1>
\sigma_2>0$, we have $\sigma_2^2 \xb_2^\intercal
\widetilde{\Bb}_{(-2)}^{-1} \xb_2<\sigma_1^2 \zb_1^\intercal
\Bb_{(-1)}^{-1} \zb_1$. Along with~(\ref{eq:lt_two}), we have
\[
\sigma_2^2 \xb_2^\intercal \Big(\sum_{l=1}^m \sigma_l^2 \xb_l\xb_l^\intercal\Big)^{-1} \xb_2 < \sigma_1^2 \zb_1^\intercal
 \Big(\sum_{l=1}^m \sigma_l^2 \zb_l\zb_l^\intercal\Big)^{-1} \zb_1.
\]
Similarly, we have $\sigma_1^2 \xb_1^\intercal \Big(\sum_{l=1}^m
\sigma_l^2 \xb_l\xb_l^\intercal\Big)^{-1} \xb_1
  > \sigma_2^2 \zb_2^\intercal \Big(\sum_{l=1}^m \sigma_l^2 \zb_l\zb_l^\intercal\Big)^{-1} \zb_2$.
Then,
\begin{eqnarray*}
&&\sigma_1^2 \zb_1^\intercal \Big(\sum_{l=1}^m \sigma_l^2 \zb_l\zb_l^\intercal\Big)^{-1} \zb_1
  +\sigma_1^2 \xb_1^\intercal \Big(\sum_{l=1}^m \sigma_l^2 \xb_l\xb_l^\intercal\Big)^{-1} \xb_1  \\
& > & \sigma_2^2 \xb_2^\intercal \Big(\sum_{l=1}^m \sigma_l^2 \xb_l\xb_l^\intercal\Big)^{-1} \xb_2
  +\sigma_2^2 \zb_2^\intercal \Big(\sum_{l=1}^m \sigma_l^2 \zb_l\zb_l^\intercal\Big)^{-1} \zb_2.
\end{eqnarray*}
Take the expectation, and we have
\begin{eqnarray}
&&E\left(\sigma_1^2 \zb_1^\intercal \Big(\sum_{l=1}^m \sigma_l^2 \zb_l\zb_l^\intercal\Big)^{-1} \zb_1\right)
  +E\left(\sigma_1^2 \xb_1^\intercal \Big(\sum_{l=1}^m \sigma_l^2 \xb_l\xb_l^\intercal\Big)^{-1} \xb_1 \right)\nonumber \\
&>& E\left(\sigma_2^2 \xb_2^\intercal \Big(\sum_{l=1}^m \sigma_l^2 \xb_l\xb_l^\intercal\Big)^{-1} \xb_2\right)
  +E\left(\sigma_2^2 \zb_2^\intercal \Big(\sum_{l=1}^m \sigma_l^2 \zb_l\zb_l^\intercal\Big)^{-1} \zb_2\right). \label{eq:both_increasing}
\end{eqnarray}
Since $\widetilde{\Omegab}$ and $\Omegab$ have the same
distribution, we have $\Omegab^\intercal\Vb \stackrel{d}{=}
\widetilde{\Omegab}^\intercal\Vb =\Omegab^\intercal\Vb_{1,2}$,
and hence $[\zb_1,\zb_2,\zb_3,\dots,\zb_n] \stackrel{d}{=}
[\xb_1,\xb_2,\xb_3,\dots,\xb_n]$. Then,
\begin{eqnarray*}
E\left(\sigma_1^2 \zb_1^\intercal \Big(\sum_{l=1}^m \sigma_l^2 \zb_l\zb_l^\intercal\Big)^{-1} \zb_1\right)
  &=&E\left(\sigma_1^2 \xb_1^\intercal \Big(\sum_{l=1}^m \sigma_l^2 \xb_l\xb_l^\intercal\Big)^{-1} \xb_1 \right)
      \nonumber \\
E\left(\sigma_2^2 \xb_2^\intercal \Big(\sum_{l=1}^m \sigma_l^2 \xb_l\xb_l^\intercal\Big)^{-1} \xb_2\right)
  &=&E\left(\sigma_2^2 \zb_2^\intercal \Big(\sum_{l=1}^m \sigma_l^2 \zb_l\zb_l^\intercal\Big)^{-1} \zb_2.\right).
\end{eqnarray*}
Therefore,  (\ref{eq:both_increasing}) becomes
\begin{equation}\label{eq:non_incrs}
E\left\{\sigma_1^2 \zb_1^\intercal \Big(\sum_{l=1}^m \sigma_l^2
\zb_l\zb_l^\intercal\Big)^{-1} \zb_1\right\}>
E\left\{\sigma_2^2 \zb_2^\intercal \Big(\sum_{l=1}^m
\sigma_l^2 \zb_l\zb_l^\intercal\Big)^{-1} \zb_2\right\}.
\end{equation}
\end{proof}

From Theorem~\ref{key}, we have
\[E\left\{\widehat \Ub_\ell \widehat\Sigmab_\ell\widehat\Vb_\ell^\top\right\}
= E(\Qb\Qb^\top\Ab)=\Ub \Lambdab\Ub^\top \Ub\Sigmab \Vb^\top= \Ub\Lambdab\Sigmab\Vb^\top.\]
The corollary below is an immediate result of Theorem~\ref{key}.
\begin{cor} By assuming all the conditions in Theorem~\ref{key}, we have the following consistency result:
\[E\left\{\widehat \Ub_\ell \widehat\Sigmab_\ell\widehat\Vb_\ell^\top\right\}
= \Ub \Db\Vb^\top,\]
where $\Db$ is an $m\times n$ diagonal matrix given by $\Db =\Lambdab\Sigmab$.
\end{cor}

\section{A numerical example}
In this section, we demonstrate the theoretical result presented in the previous section by a simple yet illustrative example. When $\Omegab$ consists of i.i.d. Gaussian entries, the phenomenon of Theorem~\ref{key} can be observed.  However, when $\Omegab$ consists of i.i.d. entries from other mean-zero, but non-Gaussian, distributions, the phenomenon of Theorem~\ref{key} can not be observed numerically. Let
\[
\Ab=\left[
\begin{array}{rrr}
3&3&3\\
-2&-2 &4\\
1&-1&0
\end{array}
\right].
\]
With singular value decomposition, we have
\[
\Ab=\left[
\begin{array}{ccc}
1&0&0\\
0&1 &0\\
0&0&1
\end{array}\right]
\, \left[
\begin{array}{ccc}
3 \sqrt{3} &0&0\\
0&2 \sqrt{6} &0\\
0&0& \sqrt{2}
\end{array}\right]
\, \left[
\begin{array}{ccc}
1/ \sqrt{3} &1/\sqrt{3}& 1/\sqrt{3}\\
-1/\sqrt{6}& -1/\sqrt{6} & 2/\sqrt{6}\\
1/\sqrt{2}&-1/\sqrt{2}& 0
\end{array}
\right].
\]
Note that the left singular matrix is an identity. Therefore, by Theorem~\ref{key}, $E (\Qb\Qb^\intercal )$ is a diagonal matrix with diagonals in the interval $(0,1)$.
To numerically check if $E (\Qb\Qb^\intercal )$ is diagonal and has diagonal entries in $(0,1)$, we approximate the expectation $E (\Qb\Qb^\intercal )$ by the average from $N$ repeated samples of $\Omegab$, $\frac1N\sum_{i=1}^N \Qb_i\Qb_i^\top$, where $\Omegab_i$ has size $3\times 2$.

\subsection{Theorem~\ref{key} is valid}
When $\Omegab$ consists of i.i.d. standard Gaussian entries, the averages $\frac1N\sum_{i=1}^N \Qb_i\Qb_i^\top$
with $N = 10^3,~ 10^4,~ 10^5,~ 10^6$ and $10^8$ are listed below, respectively.
\begin{align*}
&\left[
\begin{array}{rrr}
0.8525  &  0.0086 &  -0.0061\\
0.0086  &  0.8235  &  0.0046\\
-0.0061  &  0.0046  &  0.3239
\end{array}
\right], \qquad
\left[
\begin{array}{rrr}
 0.8440  & -0.0023  & -0.0015 \\
 -0.0023  &  0.8358  & -0.0008 \\
 -0.0015 &  -0.0008 &   0.3202 \\
\end{array}
\right],\\[2ex]
&\left[
\begin{array}{rrr}
0.8454  & -0.0003   & 0.0005 \\
-0.0003  &  0.8329  &  0.0006 \\
0.0005  &  0.0006  &  0.3217 \\
\end{array}
\right], \qquad
~\left[
\begin{array}{rrr}
0.8456 & -0.0001 & 0.0001 \\
-0.0001 & 0.8326 & 0.0001 \\
0.0001 & 0.0001 & 0.3223
\end{array}
\right],~ {\rm and}\\[2ex]
&\left[
\begin{array}{rrr}
%
0.8452  &  0.0000  & 0.0000 \\
0.0000  &  0.8323  &  0.0000 \\
0.0000  &  0.0000  &  0.3226
\end{array}
\right].
\end{align*}
They are getting closer to a diagonal matrix as $N$ increases.

\subsection{Theorem~\ref{key} is not valid}

In this subsection, we illustrate four cases that Theorem~\ref{key} is not valid numerically.

\begin{itemize}
\item {\it $\Omegab$ consists of i.i.d. entries from Uniform $(-1,1)$.}

Unlike the Gaussian case, the consistency result~(\ref{eq:EQQT}) in Theorem~\ref{key}  can not be observed numerically for $\Omegab$ being sampled from Uniform~$(-1,1)$.
The averages with $N = 10^6$ and $10^8$, respectively, are listed below.
\[ \left[
\begin{array}{rrr}
0.8377 & 0.0125 & 0.0000 \\
0.0125 & 0.8366 & -0.0005 \\
0.0000 & -0.0005 & 0.3287
\end{array}
\right]  ~~{\rm and}~~
\left[
\begin{array}{rrr}
0.8374 & 0.0127 & 0.0000 \\
0.0127 & 0.8337 & 0.0000 \\
0.0000 & 0.0000 & 0.3289
\end{array}
\right]
\]
The average fails to converge to a diagonal matrix  numerically for $N$ up to $10^8$.

\item {\it$\Omegab$ consists of i.i.d. entries from $t$-distribution with 3 degrees of freedom.}

Consider a  $t$-distribution with 3 degrees of freedom.
The averages with $N = 10^6$ and $10^8$, respectively, are listed below.
\[ \left[
\begin{array}{rrr}
0.8510 &  -0.0092  & 0.0000\\
   -0.0092 &   0.8313  &  0.0001\\
    0.0000  &  0.0001  &  0.3177
\end{array}
\right]  ~~{\rm and}~~
\left[
\begin{array}{rrr}
0.8512  & -0.0094  & 0.0000\\
-0.0094  &  0.8311 &   0.0000\\
 0.0000 &   0.0000  &  0.3177\end{array}
\right].
\]
The average fails to converge to a diagonal matrix  numerically for $N$ up to $10^8$.

\item {\it $\Omegab$ consists of i.i.d. entries from a shifted exponential.}

Consider a shifted exponential distribution  with probability density function $\exp\{-(x+1)\}$ on interval $(-1,\infty)$. This distribution is mean zero but asymmetric. The averages with $N = 10^6$ and $10^8$, respectively, are listed below.
\[ \left[
\begin{array}{rrr}
0.8698 & -0.0004 & -0.0003 \\
-0.0004 & 0.8304 & -0.0005 \\
-0.0003 & -0.0005 & 0.2998
\end{array}
\right]  ~~{\rm and}~~
 \left[
\begin{array}{rrr}
0.8696 & -0.0005 &  -0.0000 \\
   -0.0005 &   0.8305  & -0.0000 \\
   -0.0000  & -0.0000  &  0.2999
\end{array}
\right].
\]
The average fails to converge to a diagonal matrix  numerically for $N$ up to $10^8$.

\item {\it $\Omegab$ consists of i.i.d. entries from a discrete distribution.}

Consider a discrete distribution, which takes values $\{-1,1\}$ with equal probability.
The averages with $N = 10^6$ and $10^8$, respectively, are listed below.
\[ \left[
\begin{array}{rrr}
0.8803 & 0.0445 & 0.0001 \\
0.0445 & 0.8745 & -0.0003 \\
0.0001 & -0.0003 & 0.2452
\end{array}
\right] ~~{\rm and}~~
\left[
\begin{array}{rrr}
0.8797 & 0.0449 & 0.0000 \\
0.0449 & 0.8743 & 0.0000 \\
0.0000 & 0.0000 & 0.2459
\end{array}
\right].
\]
The average fails to converge to a diagonal matrix  numerically for $N$ up to $10^8$.
\end{itemize}

\section{Discussion and conclusion}

We have provided a statistical justification for the rSVD algorithm in Theorem~\ref{key} under the assumption $\Omegab$ consisting of i.i.d. Gaussian entries. Theorem~\ref{key} indicates that $\Qb$ contains consistent information of the left singular vectors $\Ub$ and its ordering through decreasing diagonals of $\Lambdab$. However, this theorem might not be valid for other types of distributions, such as uniform distribution, $t$-distribution, asymmetric shifted-to-zero-mean exponential distribution, and discrete uniform distribution over $\{\pm 1\}$. The Gaussian distribution is the only distribution that we know so far, such that Theorem~\ref{key} holds. Further study on the sampling distribution for $\Omegab$ is worth pursuing, which might play an important role in the quality of accuracy for randomized algorithms in general.

\section*{Acknowledgement}
This work is partially supported by the Ministry of Science and Technology and the National Center for Theoretical Sciences in Taiwan. 

\bibliographystyle{elsarticle-harv}
\bibliography{rsvd}

\end{document}

%% file: bold.tex
\def\Ab{{\boldsymbol A}}
\def\Bb{{\boldsymbol B}}

\def\Db{{\boldsymbol D}}

\def\Qb{{\boldsymbol Q}}

\def\Ub{{\boldsymbol U}}
\def\Vb{{\boldsymbol V}}
\def\Wb{{\boldsymbol W}}

\def\Yb{{\boldsymbol Y}}
\def\Zb{{\boldsymbol Z}}

\def\vb{{\boldsymbol v}}

\def\xb{{\boldsymbol x}}

\def\zb{{\boldsymbol z}}

\def\Lambdab{{\boldsymbol     \Lambda}}

\def\Sigmab{{\boldsymbol      \Sigma}}

\def\Omegab{{\boldsymbol      \Omega}}